\documentclass[11pt, a4paper]{article}

\usepackage[T1]{fontenc}
\usepackage[utf8]{inputenc}
\usepackage{amsmath,amssymb,amsthm, latexsym,color,epsfig,enumerate,a4, graphicx}
\usepackage{xspace}
\usepackage{bbding}
\usepackage{tikz}
\usetikzlibrary{patterns}
\usepackage{comment}
\usepackage{enumerate}
\usepackage{lmodern}
\usepackage{microtype}
\usepackage{a4wide}

\theoremstyle{plain}
\newtheorem{theorem}{Theorem}[section]
\newtheorem{lemma}[theorem]{Lemma}

\newtheorem{proposition}[theorem]{Proposition}

\newtheorem{question}[theorem]{Question}

\theoremstyle{definition}
\newtheorem{definition}[theorem]{Definition}

\newcommand{\eps}{\ensuremath{\varepsilon}}

\newcommand{\NN}{\ensuremath{\mathbb N}}

\newcommand{\calH}{\mathcal{H}}

\newcommand{\calZ}{\mathcal{Z}}

\title{On a Ramsey-T\'uran variant of the Hajnal-Szemer\'edi theorem}
\author{  
  Rajko Nenadov 
  \thanks{
  	Department of Computer Science, ETH Zurich, Switzerland. Email: {\tt rnenadov@inf.ethz.ch}. Part of this research was performed when the author was at Monash University.  
  }
  \and
  Yanitsa Pehova
  \thanks{Mathematics Institute, University of Warwick, Coventry CV4 7AL, UK. Email: \texttt{y.pehova@warwick.ac.uk}. This work has received funding from the European Research Council (ERC) under the European Union’s Horizon 2020 research and innovation programme (grant agreement No 648509). This publication reflects only its authors' view; the European Research Council Executive Agency is not responsible for any use that may be made of the information it contains. Part of this research was performed when the author was visiting Monash University.}
}
\date{}

\begin{document}
\maketitle

\begin{abstract}
A seminal result of Hajnal and Szemer\'edi states that if a graph $G$ with $n$ vertices has minimum degree $\delta(G) \ge (r-1)n/r$ for some integer $r \ge 2$, then $G$ contains a $K_r$-factor, assuming $r$ divides $n$. Extremal examples which show optimality of the bound on $\delta(G)$ are very structured and, in particular, contain large independent sets. In analogy to the Ramsey-T\'uran theory, Balogh, Molla, and Sharifzadeh initiated the study of how the absence of such large independent sets influences sufficient minimum degree. We show the following two related results:
\begin{itemize}
	\item For any $r > \ell \ge 2$, if $G$ is a graph satisfying $\delta(G) \ge \frac{r - \ell}{r - \ell + 1}n + \Omega(n)$ and $\alpha_\ell(G) =o(n)$, that is, a largest $K_\ell$-free induced subgraph has at most $o(n)$ vertices, then $G$ contains a $K_r$-factor. This is optimal for $\ell = r - 1$ and extends a result of Balogh, Molla, and Sharifzadeh who considered the case $r = 3$.

	\item If a graph $G$ satisfies $\delta(G) = \Omega(n)$ and $\alpha_r^*(G) =o(n)$, that is, every induced $K_r$-free $r$-partite subgraph of $G$ has at least one vertex class of size $o(n)$, then it contains a $K_r$-factor. A similar statement is proven for a general graph $H$.
\end{itemize} 
\end{abstract}

\section{Introduction}

Given graphs $H$ and $G$, a collection of vertex-disjoint copies of $H$ in $G$ is called an \emph{$H$-tiling}. A \emph{perfect $H$-tiling} of $G$, or an \emph{$H$-factor} for short, is an $H$-tiling which covers all the vertices of $G$. Note that a perfect matching corresponds to a $K_2$-factor, thus the notion of an $H$-factor is a natural generalisation from edges to arbitrary graphs. 

Determining sufficient conditions for the existence of an $H$-factor is one of the fundamental lines of research in Extremal Graph Theory. Textbook results of Hall and Tutte give a sufficient condition for the existence of a perfect matching (e.g.\ see \cite{diestel}). The first treatment of a more general case is by Corr\'adi and Hajnal \cite{corradi1963maximal} who showed that if a graph with $n$ vertices has minimum degree $2n/3$ then it contains a $K_3$-factor, assuming an obvious divisibility condition. This is easily seen to be the best possible bound: consider a complete tripartite graph with parts of size $n/3 - 1, n/3$ and $n/3 + 1$. The minimum degree of such a graph is $2n/3 - 1$ and as every triangle has to be traversing it does not contain a triangle-factor. Hajnal and Szemer\'edi \cite{hajnalszemeredi} showed that $\delta(G) \ge (r-1)n/r$ is sufficient for the existence of a $K_r$-factor and a similar example shows that this is optimal. Characterising the best possible bound on a minimum degree for an arbitrary graph $H$ has attracted significant attention (some of the milestones along this road include \cite{alon1996h,komlos2000tiling,komlos2001proof}) until it was finally settled by K\"uhn and Osthus \cite{kuhn2009minimum}. 

Knowing that extremal examples for these results have very special structure, a natural follow-up question is how much the bound on a minimum degree can be weakened if we exclude graphs with such a structure? Here we consider exclusion of one particular feature of  extremal examples, namely the existence of a large independent set. This was first studied by Balogh, Molla, and Sharifzadeh \cite{balogh2016} who showed that if the independence number of $G$ is $o(n)$ and $\delta(G) \ge n/2 + \eps n$, for any constant $\eps > 0$, then $G$ contains a triangle-factor. This significantly weakens the bound $\delta(G) \ge 2n/3$ from the Corr\'adi-Hajnal theorem. By considering a graph $G$ consisting of two cliques of size $n/2 - 1$ and $n/2 + 1$, we see that this bound on minimum degree is the best possible.

More generally, let $\alpha_\ell(G)$ denote the size of a largest $K_\ell$-free induced subgraph of $G$ (the independence number of $G$ corresponds to $\alpha_2(G)$). Generalising the result of Balogh, Molla, and Sharifzadeh we ask the following question.

\begin{question} \label{question}
  Let $r \ge \ell \ge 2$ be integers and let $G$ be a graph with $n$ vertices and $\alpha_\ell(G) = o(n)$. What is the best possible minimum degree condition on $G$ that guarantees a $K_r$-factor?
\end{question}

There is a clear analogy between this question and the Ramsey-T\'uran theory: a central topic in the Ramsey-T\'uran theory is to determine $\mathbf{RT}_\ell(n, H, o(n))$, the smallest number of edges which guarantees that every graph graph $G$ witn $n$ vertices with $\alpha_\ell(G) = o(n)$ contains a copy of $H$. For more on the Ramsey-T\'uran theory see \cite{erdos70,erdHos1994turan,erdHos1983more,simonovits2001ramsey}. Here we replace a T\'uran-type conclusion with one along the lines of the Hajnal-Szemer\'edi theorem.

The case $r = 3$ and $\ell = 2$ of Question \ref{question} was answered in \cite{balogh2016}. Our first main result gives some progress for the general values of $r$ and $\ell$.

\begin{theorem} \label{thm:main1}
  Let $r > \ell \ge 2$ be integers. For any $\eps > 0$ there exists $\eps' > 0$ such that if $G$ is a graph with $n$ vertices where $r$ divides $n$, $\delta(G) \ge \frac{r - \ell}{r - \ell + 1}n + \eps n$ and $\alpha_\ell(G) \le \eps' n$ then $G$ contains a $K_r$-factor.
\end{theorem} 

This shows that the minimum degree sufficient for the existence of a $K_{r - \ell + 1}$-factor is also sufficient for a $K_r$-factor if every small induced subgraph contains a copy of $K_\ell$. For $\ell = r - 1$ this becomes $\delta(G) \ge (1/2 + \eps)n$ which is easily seen to be optimal: consider a graph $G$ consisting of two disjoint complete graphs of order $n/2 - 1$ and $n/2 + 1$. Reiher and Schacht (see the appendix in \cite{balogh2016}) observed that the techniques from \cite{han2009perfect} can be used to show that $\delta(G) \ge (1/2 + \eps)n$ is sufficient for the existence of a $K_r$-factor if additionally every subset of $G$ of size $\eps' n$ induces quadratically many edges. Thus Theorem \ref{thm:main1} improves upon this by only requiring a much weaker condition $\alpha_{r-1}(G) \le \eps' n$. We discuss the optimality of Theorem \ref{thm:main1} in the case $2 \le \ell < r - 1$ in Section \ref{sec:remarks}.

Let us briefly discuss the excluded case $r = \ell$ in Theorem \ref{thm:main1}. The required minimum degree if $\alpha_r(G) = o(n)$ is clearly at most as large as if we would only know $\alpha_{r - 1}(G) = o(n)$. As the previous extremal example also shows that for $\ell = r$ we cannot go below $n/2$, we conclude that the minimum degree for $\ell = r$ is the same as for $\ell = r - 1$. Our second main result shows that if we instead require $\alpha^*_r(G) = o(n)$, where $\alpha^*_r(G)$ is defined as the smallest number such that every $r$-partite subgraph of $G$ induced by $r$ disjoint sets of size at least $\alpha_r^*(G)$ contains a copy of $K_r$, then one can take arbitrarily small degree. More generally, we prove such a statement for an arbitrary graph $H$.

\begin{definition}
  Let $H$ and $G$ be graphs. We define $\alpha_H^*(G)$ as the smallest number such that every $v(H)$-partite graph of $G$ induced by $v(H)$ disjoint subsets $V_1, \ldots, V_{v(H)} \subseteq V(G)$, with each $|V_i| \ge \alpha_H^*(G)$, contains a copy of $H$ with one vertex in each $V_i$.
\end{definition}

\begin{theorem} \label{thm:main2} 
  Let $H$ be a fixed graph with $h$ vertices. For any $\eps > 0$ there exists $\eps' > 0$ such that if $G$ is a graph with $n$ vertices where $h$ divides $n$, $\delta(G) \ge \eps n$ and $\alpha_H^*(G) \le \eps' n$, then $G$ contains an $H$-factor.
\end{theorem} 

Let us mention a connection between Theorem \ref{thm:main2} and a recent result of Balogh, Treglown, and Wagner \cite{balogh17perturbed} on $H$-factors in dense \emph{randomly perturbed graphs}. It was shown in \cite{balogh17perturbed} that if $p \ge C n^{-1/d(H)}$, where 
$$
  d(H) = \max\left\{ \frac{e(H')}{v(H') - 1} \mid H' \subseteq H, v(H') \ge 2\right\},
$$
and $G$ is a graph with $n$ vertices and $\delta(G) \ge \eps n$, then with high probability $G \cup G(n,p)$ contains an $H$-factor. Here $G(n, p)$ denotes the binomial random graph with $n$ vertices and edge probability $p$. Theorem \ref{thm:main2} implies a stronger version of this result: with high probability $G' \sim G(n,p)$ has the property that for \emph{every} graph $G$ with $\delta(G) \ge \eps n$ we have that $G \cup G'$ contains an $H$-factor (the result from \cite{balogh17perturbed} is for one fixed $G$). This follows from the fact that with high probability $\alpha_H^*(G(n,p)) < \eps' n$ for arbitrarily small $\eps'$ (which requires sufficiently large $C = C(\eps')$ in $p$). Advantages over the proof from \cite{balogh17perturbed} is that the proof of Theorem \ref{thm:main2} is significantly shorter, simpler, and avoids the use of the Regularity Lemma.

\subsection{Proof strategy}

The proofs of Theorem \ref{thm:main1} and Theorem \ref{thm:main2} use the \emph{absorbing method}. This general method was pioneered by R\"odl, Ruci\'nski and Szemer\'edi \cite{rodl2009perfect}, though it can be traced further back to Erd\H{o}s, Gy\'arf\'as, and Pyber \cite{erdHos1991vertex} and Krivelevich \cite{krivelevich1997triangle}. In our case the construction of absorbers, the heart of the proof, is based on ideas of Montgomery \cite{montgomery2014embedding}. The first author has learned this construction from discussions with Asaf Ferber. These ideas were used in a similar context, for example, by Kwan \cite{kwan2016almost} to show that a typical Steiner triple system contains a perfect matching. 

The main idea behind the absorbing method is, given a fixed graph $H$ on $h$ vertices, to find a subset $A \subseteq V(G)$, called an \emph{absorber}, such that for every $R \subseteq V(G) \setminus A$ such that $h$ divides $|R| + |A|$ and $|R| \le \xi n$, for some small $\xi > 0$, the induced subgraph $G[A \cup R]$ contains an $H$-factor. This reduces the problem to finding an $H$-tiling in $G \setminus A$ which covers all but at most $\xi n$ vertices, which is usually much simpler. 

Widely used constructions of absorbers by R\"odl, Ruci\'nski and Szemer\'edi \cite{rodl2009perfect} and H{\`a}n, Person, and Schacht \cite{han2009perfect} rely on the property of $G$ that for every subset $S \subseteq V(G)$ of size $v(H)$ there are $\Omega(n^{v(H) t})$ \emph{$S$-absorbers} of size $v(H) t$ for some $t \in \mathbb{N}$, that is, subsets $A_S \subseteq V(G) \setminus S$ of size $|A_S| = v(H) t$ such that both $G[A_S]$ and $G[A_S \cup S]$ contain an $H$-factor. In many problems on finding $H$-factors such a property indeed holds (e.g.\ see \cite{han2017exact,lo2015f,treglown2012exact}). However, as pointed out in \cite{balogh2016}, the conditions on $G$ in the case $r = 3$ and $\ell = 2$ in Theorem \ref{thm:main1} are not strong enough to guarantee this property. Consequently, the absorber construction from \cite{han2009perfect,rodl2009perfect} fails. Moreover, this turns out to be the case for an arbitrary $r$ and $\ell$: Consider a graph $G$ obtained by taking an $r$-partite complete graph with vertex classes $V_1, \ldots, V_r$ and in each $V_i$ place a graph $\Gamma$  with $|V_i|$ vertices such that $\alpha_\ell(\Gamma) = o(n)$ and $\Delta(\Gamma) = o(n)$ (see the proof of Proposition \ref{claim:lower} for the existence of such graphs). Take an arbitrary independent set $S \subseteq V_1$ of size $r$ and consider some fixed $t \in \mathbb{N}$. Any $S$-absorber $A_S \subseteq V(G) \setminus S$ of size $|A_S| = rt$ which does not contain edges of $\Gamma$ needs to intersect each $V_i$ equally. However, any $K_r$-tiling of $A_S \cup S$ has to be traversing (that is, each copy of $K_r$ contains exactly one vertex from each $V_i$), which leaves at least $r$ vertices of $(A_S\cup S)\cap V_1$ unmatched. Therefore, $A_S$ needs to contain an edge from some $V_i$, which implies an upper bound of order $o(n^{rt})$ on the number of such sets. To summarise, as soon as $\delta(G) \le (r - 1)n/r$, we cannot use constructions from \cite{han2009perfect,rodl2009perfect}.

We show that a much weaker property, namely that for every $v(H)$-subset $S$ there exists a family of $\Theta(n)$ pairwise disjoint $S$-absorbers, suffices (see Lemma \ref{lemma:absorbing}).

\paragraph{Notation.} We follow the standard graph-theoretic notation (see \cite{diestel}). In particular, given a graph $G$ and a vertex $v \in G$, we denote with $N_G(v)$ the set of neighbours of $v$ in $G$. Given a set $V$ and an integer $k \in \NN$, we denote with $\binom{V}{k}$ the family of all $k$-subsets of $V$. Throughout the paper we assume that $n$ is sufficiently large and, for brevity, avoid the use of floors and ceilings.

\section{The Absorbing Lemma} \label{sec:proof_strategy}

\begin{definition}
  Let $H$ be a graph with $h$ vertices and let $G$ be a graph with $n$ vertices.
  \begin{itemize}
    \item We say that a subset $A \subseteq V(G)$ is \emph{$\xi$-absorbing} for some $\xi > 0$ if for every subset $R \subseteq V(G) \setminus A$ such that $h$ divides $|A| + |R|$ and $|R| \le \xi n$ the induced subgraph $G[A \cup R]$ contains an $H$-factor.

    \item Given a subset $S \subseteq V(G)$ of size $h$ and an integer $t \in \mathbb{N}$, we say that a subset $A_S \subseteq V(G) \setminus S$ is \emph{$(S, t)$-absorbing} if $|A_S| = h t$ and both $G[A_S]$ and $G[A_S \cup S]$ contain an $H$-factor.
  \end{itemize}
\end{definition}

The following lemma gives a sufficient condition for the existence of $\xi$-absorbers based on abundance of disjoint $(S, t)$-absorbers.

\begin{lemma} \label{lemma:absorbing}
    Let $H$ be a graph with $h$ vertices and let $\gamma > 0$ and $t \in \NN$ be constants. Suppose that $G$ is a graph with $n \ge n_0$ vertices such that for every $S \in \binom{V(G)}{h}$ there is a family of at least $\gamma n$ vertex-disjoint $(S, t)$-absorbers. Then $G$ contains an $\xi$-absorbing set of size at most $\gamma n$, for some $\xi = \xi(h, t, \gamma)$.
\end{lemma}

The proof of Lemma \ref{lemma:absorbing} is based on ideas of Montgomery \cite{montgomery2014embedding} and relies on the existence of `robust' sparse bipartite graphs given by the following lemma. 

\begin{lemma} \label{lemma:blueprint}
	Let $\beta > 0$. Then for every $m \ge m_0$ there exists a bipartite graph $B_m$ with vertex classes $X_m \cup Y_m$ and $Z_m$ and maximum degree $40$, such that $|X_m|= m + \beta m$, $|Y_m| = 2m$ and $|Z_m| = 3m$, and for every subset $X'_m \subseteq X_m$ of size $|X'_m| = m$ the induced graph $B[X'_m \cup Y_m, Z_m]$ contains a perfect matching.
\end{lemma}

Lemma \ref{lemma:blueprint} is a corollary of \cite[Lemma 2.8]{montgomery2014embedding}. We are now ready to prove Lemma \ref{lemma:absorbing}.

\begin{proof}[Proof of Lemma \ref{lemma:absorbing}]
  From the assumption that for every $S \in \binom{V(G)}{h}$ there are $\gamma n$ disjoint $(S, t)$-absorbers we conclude that for every vertex $v \in V(G)$ there is a family of at least $\gamma n$ copies of $H$ which contain $v$ and are otherwise disjoint. Let us denote the family of sets of vertices of each such copy (without the vertex $v$) by $\mathcal{H}_v$.

  Choose a subset $X \subseteq V(G)$ by including each vertex of $G$ with probability $q = \gamma / (500 h t)$. The parameter $q$ is chosen such that the calculations work and for now it is enough to remember that $q$ is a sufficiently small constant. A simple application of Chernoff's inequality and a union bound show that with positive probability $|X| \le 2nq$ and for each vertex $v \in V(G)$ at least $q^{h - 1} |\mathcal{H}_v| / 2$ sets from $\mathcal{H}_v$ are contained in $X$. Therefore, there exists one such $X$ for which these properties hold. Let us denote the family of sets from $\mathcal{H}_v$ completely contained in $X$ with $\mathcal{H}_v'$. 

  Set $\beta = q^{h-1} \gamma / 4$ and $m = |X| / (1 + \beta)$. Let $B_m$ be a graph given by Lemma \ref{lemma:blueprint}. Choose disjoint subsets $Y, Z \subseteq V(G) \setminus X$ of size $|Y| = 2m$ and $|Z| = 3m(h-1)$ and arbitrarily partition $Z$ into subsets $\calZ = \{Z_i\}_{i \in [3m]}$ of size $h-1$. Take any injective mapping $\phi_1 \colon X_m \cup Y_m \rightarrow X \cup Y$ such that $\phi_1(X_m) = X$, and any injective $\phi_2 \colon Z_m \rightarrow \calZ$. We claim that there exists a family $\{A_e\}_{e \in B_m}$ of pairwise disjoint $(ht)$-subsets  of $V(G) \setminus (X \cup Y \cup Z)$ such that for each $e = \{w_1, w_2\} \in B_m$, where $w_1 \in X_m \cup Y_m$ and $w_2 \in Z_m$, the set $A_e$ is $(\phi_1(w_1) \cup \phi_2(w_2), t)$-absorbing. 

  Such a family can be chosen greedily. Suppose we have already found desired subsets for all the edges in some $E' \subset B_m$. These sets, together with $X \cup Y \cup Z$, occupy at most
  \begin{align*}
    |X| + |Y| + |Z| + h t |E'| &< 4m + 3m(h-1) + ht \cdot 40 |Z_m| 
    \\ &\le 4 hm + 120 h t m \le 124 htm < 240 ht nq \le \gamma n / 2
  \end{align*}
  vertices in $G$. Choose arbitrary $e = \{w_1, w_2\} \in B_m \setminus E'$. As there are $\gamma n$ disjoint $(\phi_1(w_1) \cup \phi_2(w_2),t)$-absorbers, there are at least $\gamma n / 2$ ones which do not contain any of the previously used vertices. Pick any and proceed.

  We claim that 
  $$
    A = X \cup Y \cup Z \cup \left( \bigcup_{e \in B_m} A_e \right)
  $$
  has the $\xi$-absorbing property for $\xi = \beta / (h-1)$. Consider some subset $R \subseteq V(G) \setminus A$ such that $|R| + |A| \in h \mathbb{Z}$ and $|R| \le \xi n$. As 
  $$
    |\mathcal{H}_v'| \ge q^{ht} \gamma n / 2 \ge 2 \beta n = 2 \xi n (h-1) \ge 2|R| (h-1)
  $$
  we can greedily choose a subset $A_v \in \mathcal{H}_v'$ for each $v \in R$ such that all these sets are pairwise disjoint (recall each set in $\calH_v'$ is of size $h - 1$ and forms a copy of $H$ with $v$). This takes care of vertices from $R$ and uses exactly $|R|(h-1) \le \beta m$ vertices from $X$. If $|R|(h-1) < \beta m$ then $|A| + |R| \in h \mathbb{Z}$ implies $\beta m - |R|(h-1) \in h \mathbb{Z}$, thus we can cover the remaining vertices from $X$ with disjoint copies of $H$ such that there are exactly $m$ vertices remaining. Again, $|\mathcal{H}_v'| \ge 2 \beta n > 2 \beta m$ implies that such copies of $H$ can be found in a greedy manner. 

  Let $X'$ denote the remaining vertices from $X$ and set $X_m' = \phi_1^{-1}(X')$. By Lemma \ref{lemma:blueprint} there exists a perfect matching $M$ in $B_m$ between $X_m' \cup Y_m$ and $Z_m$. For each edge $e = \{w_1, w_2\} \in M$ take an $H$-factor in $G[\phi_1(w_1) \cup \phi_2(w_2) \cup A_e]$ and for each $e \in B_m \setminus M$ take an $H$-factor in $G[A_e]$. All together, this gives an $H$-factor of $G[A \cup R]$. 
\end{proof}

\section{Proofs of main results}

The proof of Theorem \ref{thm:main2} is somewhat easier thus we use it to demonstrate an application of Lemma \ref{lemma:absorbing}.

\begin{proof}[Proof of Theorem \ref{thm:main2}]  
 Let $\gamma = \eps / 8h^2$ and $\xi = \xi(h, \gamma, t)$ be as given by Lemma \ref{lemma:absorbing}, and set $\eps' =  \min \{\xi / h, \gamma\}$.

	The proof consists of two parts: (i) show that $G$ contains an $\xi$-absorbing set $A \subseteq V(G)$, and (ii) cover all but at most $\xi n$ vertices in $V(G) \setminus A$ using vertex-disjoint copies of $H$. Note that part (ii) is almost trivial: greedily pick disjoint copies of $H$ in $V(G) \setminus A$ as long as possible. The resulting set is $H$-free thus it has to be smaller than $h \alpha_H^*(G) \le \xi n$. 

  With Lemma \ref{lemma:absorbing} at hand, to prove part (i) it suffices to show that for an arbitrary 
  $S \in \binom{V(G)}{h}$ there is a family of $\gamma n$ vertex-disjoint $(S, t)$-absorbers for some $t\in \NN$. In this case we take $t=h$. First, for each $w \in S$ choose a subset $N_w \subseteq N_G(w) \setminus S$ of size $\eps n / (2h)$ such that all these sets are pairwise disjoint. From $h \alpha_H^*(G) < \eps n / 4h$ we have that $G[N_w]$ contains a family of $\eps n / 4h^2$ disjoint copies of $H$. Let $V_w$ be a set containing one vertex from each such copy of $H$. Each copy of $H$ which traverses all $V_w$'s, that is, it contains a vertex from each $V_w$ for $w \in S$, forms an $(S, h)$-absorber with the corresponding copies of $H$ from $N_w$'s. Greedily pick such disjoint traversing copies of $H$ one after the other, again as long as possible. As long as we have at least $|V_w|/2 > \alpha_H^*(G)$ unused vertices in each $V_w$, that is we have found less than $|V_w|/2$ traversing copies of $H$ so far, the process continues. This way we construct a family of at least $|V_w| / 2 \ge \gamma n$ disjoint $(S, h)$-absorbers, which finishes the proof.
\end{proof}

The proof of Theorem \ref{thm:main1} follows along the same lines as the proof of Theorem \ref{thm:main2}, though it it slightly more technically involved. In particular, the covering part does not come for free, thus we handle it with with the following lemma.

\begin{lemma} \label{lemma:cover}
  Let $r > \ell \ge 2$ be integers. For any $\eps, \xi > 0$ there exists $\eps' > 0$ such that if $G$ is a graph with $n \ge n_0$ vertices, $\delta(G) \ge \frac{(r - \ell)n}{r - \ell + 1} + \eps n$ and $\alpha_\ell(G) \le \eps' n$, then for every $A \subseteq V(G)$ of size at most $\eps n / 2$, there exists a $K_r$-tiling in $G \setminus A$ which covers all but at most $\xi n$ vertices.
\end{lemma}

The proof of Lemma \ref{lemma:cover} is fairly standard, however as it involves Szemer\'edi's regularity lemma we postpone it until the end of the section. We now show how together with Lemma \ref{lemma:absorbing} it implies Theorem \ref{thm:main1}.

\begin{proof}[Proof of Theorem \ref{thm:main1}]
  To prove Theorem \ref{thm:main1}, it suffices to show that for every subset $S \in \binom{V(G)}{r}$ there exists a family of $\gamma n$ vertex-disjoint $(S, t)$-absorbers for some $t\in \NN$. In particular, we do that for some $\gamma < \eps/2$ and $t=r$. Then by Lemma \ref{lemma:absorbing} there exists a $\xi$-absorbing subset $A \subseteq V(G)$ of size at most $\eps n / 2$ and by Lemma \ref{lemma:cover} there exists a $K_r$-tiling  in $G \setminus A$ covering all but at most $\xi n$ vertices. All together this implies the existence of a $K_r$-factor in $G$.

  Consider some $S \in \binom{V(G)}{r}$. Partition randomly $V(G) \setminus S$ into $r+1$ sets denoted by $V_1, \ldots, V_{r+1}$. Each $V_i$ is of size $(n - r)/(r+1)$ and by Chernoff's inequality and union bound, with high probability every vertex has at least 
  $$
    \left( \frac{r - \ell}{r - \ell + 1} + \eps/2 \right) \frac{n}{r+1}
  $$
  neighbours in each $V_i$. Therefore, there exists a partition $V_1, \ldots, V_{r+1}$ for which this holds. 

  Let us enumerate the vertices in $S$ as $v_1, \ldots, v_r$. We show that for every $X_i \subseteq V_i$ of size at most $\eps n / 4(r+1)$ there exists a copy of $K_{r}$ in $V_{r+1}$, with vertices labelled $w_1, \ldots, w_r$, and a copy of $K_{r-1}$ in $N_G(v_i) \cap N_G(w_i) \cap (V_i \setminus X_i)$ for every $i \in \{1, \ldots, r\}$. Note that such copies of $K_{r-1}$ together with the copy of $K_r$ in $V_{r+1}$ form an $(S, r)$-absorber. This allows us to greedily form a family of $\eps n / 4(r+1)r$ disjoint $(S, r)$-absorbers, which finishes the proof.

  The previous claim follows almost trivially from the bound on the minimum degree and $\alpha_\ell(G) \le \eps' n$. First, note that each vertex has at least
  $$
     \left( \frac{r - \ell}{r - \ell + 1} + \eps/4 \right) \frac{n}{r+1}
  $$
  neighbours in each $V_i \setminus X_i$. As $|V_i| < n/(r+1)$ this implies that any set of $r - \ell + 1$ vertices has a common neighbourhood of size at least $\eps n / 4(r+1)$ in $V_i \setminus X_i$. This means we can start with an arbitrary vertex $w_1 \in V_{r+1} \setminus X_{r+1}$ and iteratively for $2 \le i \le r - \ell$ pick a vertex $w_i$ which is in the neighbourhood of $w_1, \ldots, w_{i-1}$. Such vertices form $K_{r - \ell}$ and as $\alpha_\ell(G) \le \eps n / 4(r+1)$ there exists a copy of $K_\ell$ in their common neighbourhood in $V_{r+1} \setminus X_{r+1}$. This gives us a copy of $K_r$ in $V_{r+1} \setminus X_{r+1}$. Now for each $i \in [r]$  repeat a similar argument in order to find a copy of $K_{r-1}$ in $N_G(v_i) \cap N_G(w_i) \cap (V_i \setminus X_i)$. We omit the details.
\end{proof}

\subsection{Proof of Lemma \ref{lemma:cover}} 

The proof of Lemma \ref{lemma:cover} is based on a standard application of the regularity method (e.g. see \cite{komlos96}). We quote below some basic definitions and a statement of the Regularity Lemma.

\begin{definition}
  Given a graph $G$ and disjoint subsets $A, B \subseteq V(G)$, we say that the pair $(A, B)$ is \emph{$\mu$-regular} if for all $X\subseteq A$, $|X|\geq \mu |A|$ and $Y\subseteq B$, $|Y|\geq \mu |B|$ we have
  \[\left|d(X,Y)-d(A,B)\right|\leq \mu\]
  where $d(X,Y)=e(X,Y)/|X||Y|$ and $e(X, Y)$ is the number of edges in the bipartite subgraph of $G$ induced by $X$ and $Y$.
\end{definition}

The following version of the Regularity Lemma is known as the \emph{``degree form''} (see \cite[Theorem 1.10]{komlos96}).

\begin{lemma} \label{lem:rl}
  For every $\mu > 0$ there is an $N = N(\mu)$ such that if $G$ is a  graph with $n$ vertices and $d \in [0,1]$ is any real number, there exists a partition $V(G)=V_0\cup...\cup V_k$ and a spanning subgraph $G'\subseteq G$ with  the following properties:
  \begin{enumerate}[(a)]
    \item $k \leq N$;
    \item $|V_0| \leq \mu n$;
    \item $|V_i|=m$ for all $1\leq i\leq k$, where $m\leq \mu n$;
    \item $d_{G'}(v)>d_{G}(v)-(d+\mu)|G|$ for all $v\in V(G)$;
    \item all $V_i$ are independent sets in $G'$;
    \item all pairs $(V_i,V_j)$ are $\mu$-regular in $G'$ with density 0 or at least $d$.
  \end{enumerate}
\end{lemma}

We only give a sketch of the proof of Lemma \ref{lemma:cover} as it is a straightforward application of the regularity method.

\begin{proof}[Sketch of the proof of Lemma \ref{lemma:cover}] 

Let $\mu \ll \epsilon,\xi$, that is $\mu$ is sufficiently smaller than $\epsilon$ and $\xi$, and $d=\epsilon/4$. Apply the Regularity Lemma (Lemma \ref{lem:rl}) on $G$ with $\mu$ and $d$ to obtain a partition $V_0,...,V_k$ of $V(G)$ and a spanning subgraph $G'\subseteq G$ with the properties (a)--(f) as stated. Let $R$ be the \emph{reduced graph} of this partition: $R$ has vertex set $\{1, \ldots, k\}$ and there is an edge between $i$ and $j$ if and only if the pair $(V_i, V_j)$ has density at least $d$ in $G'$. Property (d) implies that $R$ has minimum degree at least
$$
  \left(\frac{r - \ell}{r - \ell + 1} + \eps/2 \right) k,
$$
thus by the Hajnal-Szemer\'edi theorem it contains a $K_{r-\ell+1}$-tiling which covers all but at most $r - \ell$ vertices (in case $k$ is not divisible by $r - \ell + 1$). For the rest of the proof we ignore $V_0$ and $V_i$'s for $i \in [k]$ which corresponds to vertices not covered by such a tiling. This way we ignore at most $\mu n + (r - \ell) \mu n < \xi n / 2$ vertices. 

Consider one of the copies of $K_{r - \ell + 1}$ in the obtained tiling in $R$. Without loss of generality we may assume that it corresponds to vertex classes $V_1, \ldots, V_{r - \ell + 1}$. We show that we can find a $K_r$-tiling in $G[V_1 \cup \ldots V_{r - \ell + 1}]$ which covers all but at most $\xi m / 2$ vertices in each $V_j$. Applying this to every copy of $K_{r - \ell + 1}$ from the tiling of $R$ we find a $K_r$-tiling of $G$ covering all but at most $\xi n$ vertices, as desired.

To show that there exists a $K_r$-tiling in $G[V_1 \cup \ldots V_{r - \ell + 1}]$ which covers all but at most $\xi m / 2$ vertices in each $V_j$, it suffices to show that for any $z \in [r - \ell + 1]$ and any choice of subsets $V_j' \subseteq V_j$ of size $|V_j'| \ge \xi m / 8$ for $j \in [r - \ell + 1]$, there exists a copy of $K_r$ in $G[V_1' \cup \ldots \cup V_{r - \ell + 1}']$ with exactly one vertex in each $[r - \ell + 1] \setminus z$ and $\ell$ vertices in $V_z'$. By repeatedly applying this $(1 - \xi/4) m / r$ times for each $z \in [r - \ell + 1]$, each time removing vertices from the obtained $K_r$, we obtain the desired $K_r$-tiling. 

Consider some subsets $V_j' \subseteq V_j$ for $j \in [r - \ell + 1]$ such that $|V_j'| \ge \xi m / 8$. By the Slicing Lemma (see \cite[Fact 1.5]{komlos96}) each pair $(V_i', V_j')$ is $\mu'$-regular for $\mu' = 8\mu / \xi$. Without loss of generality we may assume $z = r - \ell + 1$. Our goal is to find a vertex $w_j \in V_j'$ for each $1 \le j \le r - \ell$ such that these vertices form $K_{r - \ell}$ and their common neighbourhood $N_z \subseteq V_z'$ in $V_z'$ is of size at least 
$$
  |N_z| \ge (d/2)^{r - \ell} |V_z'| \ge (d/2)^{r - \ell} \xi m / 8 \ge (d/2)^{r - \ell} \xi n / 16N \ge \eps' n, 
$$
for sufficiently small $\eps'$ (recall that $N$ is a constant). The proof of this can be found, for example, in a textbook by Diestel \cite[Theorem 7.5.2]{diestel}, thus we omit it. Finally, as $\alpha_\ell(G) \le \eps' n$ we can find a copy of $K_\ell$ in $G[N_z]$ which completes a desired copy of $K_r$. 
\end{proof}

\section{Concluding remarks} \label{sec:remarks}


As far as we are aware, it is not known whether Theorem \ref{thm:main1} is optimal for $r - 1> \ell \ge 2$. In particular, a construction from \cite{balogh2016} shows that $\delta(G) \ge (r - 2)n / r$ is necessary for $\ell = 2$, while Theorem \ref{thm:main1} shows that $(r - 2)n/(r-1) + o(n)$ is sufficient. We now generalise their construction to arbitrary $r > \ell \ge 2$. 

\begin{proposition}\label{claim:lower}
If $r > \ell \ge 2$ then there exists $n_0\in \NN$ such that for every $n \ge n_0$ there is a graph $G$ on $n$ vertices such that
$$
  \delta(G) \ge 
  \begin{cases}
    \frac{r - \ell}{r}n, &\text{ if } \ell \le r/2 \\
    1/2, &\text{ otherwise},
  \end{cases}
$$
and $\alpha_\ell(G) = o(n)$ which does not contain a $K_r$-factor.
\end{proposition}
\begin{proof}
Let $\Gamma_\ell(n)$ be a $K_{\ell+1}$-free graph with $n \ge n_0$ vertices, $\alpha_\ell(G) = o(n)$ and maximum degree $o(n)$. Such a graph can be obtained, for example, by considering a random graph $G(n,p)$ with edge probability $p = n^{-2/(\ell+1)}$: By applying the FKG inequality we have
\begin{align*}
  \Pr[G(n,p) \text{ is } K_{\ell+1}\text{-free}] &\ge \prod_{S \in \binom{[n]}{\ell+1}} \Pr[G(n,p)[S] \neq K_{\ell+1}] \\
  &\ge \left(1 - p^{\binom{\ell+1}{2}}\right)^{\binom{n}{\ell+1}} \ge \exp \left( -2p^{\binom{\ell+1}{2}} n^{\ell+1} \right) = \exp (-2n).
\end{align*}
In the third inequality we used $1 - x > e^{-2x}$ for $0 < x < 1/2$. On the other hand, for $p \ge n^{-2/\ell + \eps}$ for some $\eps > 0$, we have that a fixed subset of $G(n,p)$ of size $n^{1 - \eps/2}$ contains $K_\ell$ with probability $1 - \exp(\Omega(n^{2 - \eps} p))$. A union bound over all subsets of such size shows that $\alpha_\ell(G(n,p)) \le n^{1 - \eps/2}$ with probability $1 - 2^n \exp(\Omega(n^{2-\eps} p))$. By choosing sufficiently small $\eps > 0$, with positive probability we have both $\alpha_\ell(G(n,p)) = o(n)$ and $G(n,p)$ is $K_{\ell+1}$-free. 

Having graphs $\Gamma_\ell(n)$ at hand, we can finish the proof of the claim. As mentioned before, the construction we present is a straightforward generalisation of the construction from \cite{balogh2016}. Consider some $2 \le \ell \le r/2$ and let $r = x \ell + y$ for some $x, y \in \NN$ and $1 \le y \le \ell$. We create a graph $G$ by taking an $(x+1)$-partite complete graph with one set $V_1$ of size $y n/r - 1$, one set $V_2$ of size $\ell n/ r + 1$ and all other sets $V_3, \ldots, V_{x+1}$ of size $\ell n / r$, and within each set $V_i$  put the graph $\Gamma_\ell(|V_i|)$. Such a graph has minimum degree $y n / r - 1 + (x - 1)\ell n / r = (r - \ell) n / r - 1$. Because $V_i$ does not contain $K_{\ell+1}$, any $K_r$ in such a graph $G$ has to contain at least $y$ vertices from $V_1$ and cannot contain more than $\ell$ from any other set. Therefore, a $K_r$-tiling can have at most $\lfloor |V_1|/y \rfloor < n/r$ copies of $K_r$, which is not enough to cover all the vertices in $V_2$.  
\end{proof}

To summarise, Theorem \ref{thm:main1} and Proposition \ref{claim:lower} give an upper and lower bound on Question \ref{question}. The difference between these bounds gets larger as $\ell$ goes from $2$ to $r/2$ and then decreases again as it goes further to $r - 1$. For $\ell = r - 1$ Theorem \ref{thm:main1} matches Proposition \ref{claim:lower}, thus resolving Question \ref{question} in this case. It would be very interesting to determine the correct minimum degree condition for all other cases. This adds to a list of open problems posed in \cite{balogh2016}.

Finally, it is worth mentioning that Balogh et al.\ \cite{balogh2018triangle} showed that if $\alpha_2(G) = o(n)$ and $\delta(G) \ge n/3 + o(n)$ then there exists a triangle-tiling covering all but at most $4$ vertices. Moreover, in the case when $3$ divides $n$ they characterise a `barrier' for the existence of a triangle-factor. In the same spirit, as a step towards answering Question \ref{question} it would be interesting to show that $\delta(G) \ge (r - 2)n/r + o(n)$ is sufficient for the existence of a $K_r$-tiling covering all but constantly many vertices for $r \ge 4$. The construction from Proposition \ref{claim:lower} can be modified to show that there exists a graph $G$ with $\alpha_2(G) = o(n)$ and $\delta(G) = (r-2)n/r - \eps n$ such that no $K_r$-tiling of $G$ covers more than $(1 - \eps)n$ vertices.

\paragraph{Acknowledgement.} The authors would like to thank Anita Liebenau for the valuable discussions.

\bibliographystyle{abbrv}
\bibliography{references}

\end{document}